\documentclass[12pt]{article}

\usepackage{fullpage,indentfirst}
\usepackage{amsfonts,amssymb,stmaryrd}
\usepackage{diagrams}

\usepackage{pstricks}
\usepackage{pst-node}

\newcounter{paragrafsubsub}[subsubsection]
\renewcommand{\theparagrafsubsub}{%
\thesubsubsection.\arabic{paragrafsubsub}}
\newcommand{\paragrafsubsub}{%
\ifcase\value{paragrafsubsub} {}\else \vspace{0.2cm}\fi
\refstepcounter{paragrafsubsub}
{\bf \theparagrafsubsub}\hspace{0.2em}--- }

\newcounter{paragrafsub}[subsection]
\renewcommand{\theparagrafsub}{\thesubsection.\arabic{paragrafsub}}
\newcommand{\parag}{%
\ifcase\value{paragrafsub} {}\else \vspace{0.2cm}\fi
\refstepcounter{paragrafsub}
{\bf \theparagrafsub}\hspace{0.2em}--- }

\newcounter{paragraf}[section]
\renewcommand{\theparagraf}{\thesection.\arabic{paragraf}}
\newcommand{\paragr}{%
\ifcase\value{paragraf} {}\else \vspace{0.2cm}\fi
\refstepcounter{paragraf}
{\bf \theparagraf}\hspace{0.2em}--- }

\newcommand\paragraphe{%
\ifcase\value{subsection} %
  \paragraf 
\else 
\ifcase\value{subsubsection}\paragrafsub %
 \else\paragrafsubsub\fi 
\fi}

\newtheorem{prop}{Proposition}[section]
\newtheorem{lemma}{Lemma}[section]
\newtheorem{coro}{Corollary}[section]
\newenvironment{proof}{\noindent {\sc Proof.}}{\hfill$\square$}
\newtheorem{theointro}{Theorem}

\newenvironment{definition}{\noindent {\sc Definition.}}{}

\newarrow{Dble}=====
\newarrow{Tple}33333

\def\PP{\mathbb{P}}\def\QQ{\mathbb{Q}}
\def\FF{\mathbb{F}}

\def\Chi{{\cal X}}
\def\rk{{\rm rk}}
\def\GB{{\mathcal B}}
\def\GP{{\mathcal P}}
\def\a{\alpha}

\def\O{{\mathcal{O}}}

\def\longto{\longrightarrow}

\def\dim{{\rm dim}}

\def\longto{\longrightarrow}

\def\TH{{T(H)}}
\def\lg{{\mathfrak{g}}}
\def\lh{{\mathfrak{h}}}
\def\lsl{{\mathfrak{sl}}}
\def\lsp{{\mathfrak{sp}}}
\def\lso{{\mathfrak{so}}}

\def\SO{{\rm SO}}
\def\PSO{{\rm PSO}}
\def\PSL{{\rm PSL}}
\def\PSp{{\rm PSp}}

\def\rhobar{\rho_\phi}
\def\rhodelta{\rho_\Delta}
\def\dyn{\mathcal{D}}
\def\KK{\mathbb{K}}

\def\GB{{\cal B}}
\def\H{{\bf H}}
\def\TT{{\mathbb T}}

\psset{nodesep=0,offset=0,fillcolor=black,unit=0.7cm,labelsep=0.4}
\def\supdiag{
\begin{pspicture}(0,-0.25)(0.4,0.25)
    \psline(0.4,-0.25)(0,0)(0.4,0.25)
\end{pspicture}
}

\begin{document}
\begin{center}
  {\large\bf
Spherical homogeneous spaces of minimal rank}\\
\vspace{3mm}
N. Ressayre\\
\end{center}

\noindent
{\bf Abstract.}
Let $G$ be a complex connected reductive algebraic group
and $G/B$ denote the flag variety of $G$.  
A $G$-homogeneous space $G/H$ is said to be {\it spherical}  if $H$ has a 
finite number of orbits in $G/B$.
A class of spherical homogeneous spaces containing the tori, the complete 
homogeneous spaces and the group $G$ (viewed as a $G\times G$-homogeneous space) 
has particularly nice proterties.
Namely, the pair $(G,H)$ is called a {\it spherical pair of minimal rank} if 
there exists $x$ in $G/B$ such that the orbit $H.x$ of $x$ by $H$ 
is open in $G/B$ and the stabilizer $H_x$ of $x$ in $H$ contains a maximal 
torus of $H$. 
In this article, we study and classify the spherical pairs of minimal rank.

\section{Introduction}
Let $G$ be a complex connected reductive algebraic group.
Let $\GB$ denote the flag variety of $G$.  
Let $H$ be an algebraic subgroup of $G$ which has a finite number of orbits
in $\GB$ ; the subgroup $H$ and the homogeneous space $G/H$ are said 
to be {\it spherical}.

In this article, we study and classify a class of spherical homogeneous spaces
containing the tori, the complete homogeneous spaces and the group $G$ viewed 
as a $G\times G$-homogeneous space.
Namely, the pair $(G,H)$ is called a {\it spherical pair of minimal rank} if 
there exists $x$ in $\GB$ such that the orbit $H.x$ of $x$ by $H$ 
is open in $\GB$ and the stabilizer $H_x$ of $x$ in $H$ contains a maximal 
torus of $H$.
In \cite{Kn:WGH} the rank $\rk(G/H)$ of the homogeneous space $G/H$ is defined.
Moreover, we have $\rk(G/H)\geq \rk(G)-\rk(H)$ (where $\rk(G)$ and
$\rk(H)$ denotes the ranks of the groups $G$ and $H$) with equality if and only
if $(G,H)$ is of minimal rank. This explains the name.
The spherical pairs $(G,H)$ of minimal rank such that $H$ is a symmetric 
subgroup of $G$ firstly appear in \cite{Br:eqvectbun}.
During the redaction of this article the compactifications of 
the spherical homogeneous spaces  of minimal rank was studied in 
\cite{Tchoudjem:rngmin} and \cite{BrionJoshua:rngmin}.\\

Let us state our main result.
Propositions~\ref{prop:redHred}, \ref{prop:redHssc} and
\ref{prop:redHsimple} reduce the classification to the special case 
when $G$ is semi-simple adjoint and $H$ is simple.
Indeed, any spherical pair of minimal rank is obtained from special ones and
toric ones by products, finite covers and parabolic inductions.
Next, we prove 

\begin{theointro}
\label{th:introclass}
  The spherical pairs $(G,H)$ of minimal rank with $G$ semi-simple adjoint
and $H$ simple are:
\begin{enumerate}
\item $G=H$.
\item $H$ is simple and diagonally embedded in $G=H\times H$.
\item $(\PSL_{2n},\PSp_{2n})$ with $n\geq 2$.
\item $(\PSO_{2n},\SO_{2n-1})$ with $n\geq 4$.
\item $(SO_7,G_2)$.
\item $(E_6,F_4)$.
\end{enumerate}
\end{theointro}

We denote by $\H(\GB)$ the set of the $H$-orbit closures in $\GB$.
If $H=P$ is a Borel subgroup of $G$, the elements of $\H(\GB)$ are 
the famous Schubert varieties.
Most of combinatorial and geometric properties of the Schubert varieties cannot
be generalized to the elements of $\H(G/B)$ if $H$ is only spherical.
However, if $H$ is of minimal rank the elements of $\H(G/B)$ have 
nice properties. Let us give details.

The Weyl group $W$ of $G$ acts transitively on the set of Schubert varieties; 
this action parametrizes these varieties by $W/W_P$.
In general, F.~Knop has defined in \cite{Kn:WGH} an action of $W$ in $\H(\GB)$;
but, it seems to be difficult to deduce a parametrization of 
$\H(\GB)$ from this action.
We show in Proposition~\ref{prop:defequivHGB} that $G/H$ is of minimal
rank if and only if the action of $W$ is transitive on $\H(\GB)$. 
In this case, the  isotropy groups 
are isomorphic to the Weyl group $W_H$ of $H$; and, $W/W_H$ parametrizes
$\H(\GB)$.

The Schubert varieties are normal; but, in general elements of $\H(\GB)$
are not normal (see~\cite{Br:GammaGH} or \cite{Spin:these} for examples).
By a result of Brion, if $G/H$ is of minimal rank, the elements of
$\H(\GB)$ are normal.

In \cite{Kn:WGH}, F.~Knop also defined an action of a monoid $\tilde{W}$ 
(constructed from the generators of $W$) on  $\H(\GB)$.
Moreover, the inclusion defines an order on $\H(\GB)$ which generalizes the
Bruhat order for the Schubert varieties.
The description of the Bruhat order from the action of  $\tilde{W}$  is 
well known as the cancellation lemma. 
In general, no such description of this order is known.
Corollary~\ref{cor:cancel} is a cancellation lemma in the minimal rank case.

The number of Schubert varieties of dimension $d$ equal those of 
the codimension $d$.
In Proposition~\ref{prop:sym} we show such a symmetry property of $\H(\GB)$ 
for any spherical pair $(G,H)$ of minimal rank.\\

Let us explain another important motivation for this work.
Let $T$ be a maximal torus of $G$ and $X$ be a $G$-equivariant embedding
of a spherical homogeneous space $G/H$ of minimal rank.
In Proposition~\ref{prop:embcaract}, 
we show that for all fixed point $x$ of $T$ in $X$, $G.x$ is complete.
This property seems to play a key role in several works about the embeddings
of $G\times G/G$ (see for example, \cite{Tchou}).

In Section~\ref{sec:def}, we study the properties of $\H(\GB)$ and of the
toroidal embeddings of $G/H$ for the spherical pairs $(G,H)$ of minimal rank.
This allows us to give several characterizations of the minimal rank property.
In Section~\ref{sec:red}, we reduce the classification to the case when $G$
and $H$ are semisimple.
In Section~\ref{sec:class}, we classify such pairs by associating to $(G,H)$ 
an involution on the vertices of the Dynkin diagram of $G$.

\section{Equivalent definitions and first properties}
\label{sec:def}

\subsection{Minimal rank and orbits of $H$ in $\GB$}

\parag
Let us fix some general notation. If $X$ is a variety, $\dim(X)$ denotes 
the dimension of $X$.
It $x$ belongs to $X$, $T_xX$ denotes the Zariski-tangent space of $X$ at $x$.
If $\Gamma$ is an algebraic group a {\it $\Gamma$-variety} $X$ 
is a variety endowed  with an algebraic action of $\Gamma$.
Let $\Gamma$ be an affine algebraic group and $X$ be a  $\Gamma$-variety.
For $x$ a point in $X$, we denote by $\Gamma_x$ the isotropy group of $x$ and
by $\Gamma.x$ its orbit. We denote by $X^\Gamma$ the set of fixed 
point of $\Gamma$ in $X$.
We denote by $\Gamma^u$ the unipotent radical of $\Gamma$.

\parag
Let us recall that $G$ is a connected complex reductive group,
$\GB$ its flag variety, $H$ a spherical subgroup of $G$ and 
$\H(\GB)$ the set of the $H$-orbit closures in $\GB$. 
If $V$ belongs to $\H(\GB)$, we denote by $V^\circ$ the unique open 
$H$-orbit in $V$.

We recall the definition of \cite{GammaGH} of a graph $\Gamma(G/H)$
whose vertices are the elements of $\H(\GB)$. The original
construction of $\Gamma(G/H)$ due to M. Brion is equivalent but 
very slightly different (see \cite{Br:GammaGH}).

Consider the set $\Delta$ of conjugacy classes of minimal non solvable
parabolic subgroups of $G$. 
If $\a$ belongs to $\Delta$, we denote by $\GP_\a$ the $G$-homogeneous
space with isotropy $\a$. 
Then, there exists a unique $G$-equivariant map
$\phi_\a\,:\,\GB\longto\GP_\a$ which is a $\PP^1$-bundle. 

Let $V\in\H(\GB)$ and $\a\in\Delta$. 
We assume that the restriction of $\phi_\a$ to $V^\circ$ is finite and 
we denote  its degree by $d(V,\alpha)$. 
Then, $\phi_\a^{-1}(\phi_\a(V))$ is an element denoted $V'$ of $\H(\GB)$; 
in this case, we say that {\it $\a$ raises $V$ to $V'$}.  
One of the three following cases occurs.
\begin{itemize}
\item Type $U$: $H$ has two orbits in $\phi_\a^{-1}(\phi_\a(V^\circ))$
($V^\circ$ and $V'^\circ$) and $d(V,\a)=1$.
\item Type $T$: $H$ has three orbits in $\phi_\a^{-1}(\phi_\a(V^\circ))$
 and $d(V,\a)=1$. 
\item Type $N$: $H$ has two orbits in $\phi_\a^{-1}(\phi_\a(V^\circ))$
($V^\circ$ and $V'^\circ$) and $d(V,\a)=2$.
\end{itemize}

\begin{definition}
Let $\Gamma(G/H)$ be the oriented graph with vertices the elements of
$\H(\GB)$ and edges labeled by $\Delta$, where $V$ is joined to $V'$
by an edge labeled by $\a$ if $\a$ raises $V$ to $V'$. This edge is
simple (resp. double) if $d(V,\a)=1$ (resp. 2). 
Following the above cases, we say that an edge has {\it type $U$,
  $T$ or $N$}.
\end{definition}

\parag
Let us fix a Borel subgroup $B$ of $G$.
Let $Y$ be a $B$-variety.
The {\it character group $\Chi(Y)$} of $Y$ is the set of all
characters of $B$ that arise as weights of eigenvectors of $B$ in the
function field $\KK(Y)$.  Then $\Chi(Y)$ is a free abelian group of
finite rank $\rk(Y)$, {\it the rank of $Y$} (see \cite{Kn:WGH}).   
It is well known that a $B$-orbit $\O$ is isomorphic as a variety to
$\KK^l\times (\KK^*)^r$ where $r=\rk(\O)$ and 
$l=\dim(\O)-\rk(\O)$.

If $V$ belongs to $\H(\GB)$, we set:
$$
V_H=\{gH/H\,:\,g^{-1}B/B\in V\}.
$$ 
Then, $V_H$ is a $B$-orbit closure in $G/H$. Moreover, the map 
$V\longmapsto V_H$ is a bijection from $\H(\GB)$ onto the set
of the  $B$-orbit closures  in $G/H$. The rank of $V_H$ is also denoted 
by $\rk(V)$ and called {\it the rank of V}.

\parag
Let $T$ be a maximal torus of $B$. 
Let $W$ denote the Weyl group of $T$.
Every $\a$ in $\Delta$ has a unique representative $P_\a$ which
contains $B$. 
Moreover, there exists a unique $s_\a$ in $W$ such that $Bs_\a B$ is
dense in $P_\a$; and this $s_\a$ is a simple reflexion of $W$. 
The map, $\Delta\longto W,\,\a\longmapsto s_\a$ is a bijection from
$\Delta$ onto the set of simple reflexions of $W$.

F.~Knop defined in  \cite{Kn:WGH} an action of $W$ on the set $\H(\GB)$ 
by describing  the action of the $s_\a$, for any $\a\in\Delta$:

\begin{itemize}
\item Type $U$: $s_\a$ exchanges the two vertices of an edge of type
  $U$ labeled by $\a$.
\item Type $T$: If $\a$ raises $V_1$ and $V_2$ to $V$, then
  $s_\a V_1=V_2$ and $s_\a V=V$.
\item Type $N$: $s_\a$ fixes the two vertices of a double edge
  labeled by $\a$.
\item $s_\a$ fixes all others vertices of $\Gamma(G/H)$.
\end{itemize}

\parag
We can now characterize the spherical pairs of minimal rank in terms 
of $\H(\GB)$:

\begin{prop}
\label{prop:defequivHGB}
With above notation, the following are equivalent:
\begin{enumerate}
\item \label{ass:defequivHGB0}There exists $x\in\GB$ such that
$H.x$ is open in $\GB$ and $H_x$ contains a maximal torus of $H$.
\item \label{ass:defequivHGB1}$\rk(G)-\rk(H)=\rk(G/H)$.
\item \label{ass:defequivHGB2}
All the elements of $\H(\GB)$ have same rank.
\item \label{ass:defequivHGB3}All the edges in $\Gamma(G/H)$ have type $U$.
\item \label{ass:defequivHGB4}$W$ acts transitively on $\Gamma(G/H)$.
\end{enumerate}

If $(G,H)$ satisfies these properties, we say that $(G,H)$ is of minimal rank.
\end{prop}

\begin{proof}
The equivalence between the two first assertions follows 
from \cite[Corollary 3.1]{actionKnop}.

Let us recall some properties of the graph $\Gamma(G/H)$ 
from \cite{Br:GammaGH}.
If $\alpha$ raises $V$ to $V'$ by an edge of type $U$ (resp. $T$ or $N$)
then $\rk(V')=\rk(V)$ (resp. $\rk(V')=\rk(V)+1$).
Moreover, all $V$ in $\H(\GB)$ is joined to $\GB$ by an increasing path in 
$\Gamma(G/H)$ (property of connectedness).
Finally, the rank of a closed $H$-orbit in $\GB$ equals $\rk(G)-\rk(H)$.

Now, one easily checks the equivalence between 
Assertions~\ref{ass:defequivHGB1}, 
\ref{ass:defequivHGB2}, \ref{ass:defequivHGB3}  
and~\ref{ass:defequivHGB4}.
\end{proof}

\parag
Let $(G,H)$ be a spherical pair of minimal rank.
Then, the elements of $\H(\GB)$ can be parametrized.
Indeed, let $W_0$ be the stabilizer of $\GB$ for the action of $W$.
In \cite{actionKnop}, it is shown that $W_0$ is isomorphic to the Weyl group 
$W_H$ of $H$. Moreover, Proposition~\ref{prop:defequivHGB} shows that 
the Knop's action gives a bijection between $W/W_0$ and $\H(\GB)$.
In particular, we have: $|\H(\GB)|=\frac{|W|}{|W_H|}$, where $|E|$ denotes
the cardinality of the finite set $E$.\\

Each orbit closure $V$ of $H$ in $\GB$ is multiplicity-free in the sense of 
\cite{Br:GammaGH}.
In particular, by \cite[Theorem~5]{Br:GammaGH} $V$ is normal.

\parag
In this paragraph, we are interested in reading the generalized Bruhat order 
off the graph $\Gamma(G/H)$.
Let us start by showing the following nice property of this graph: 

\begin{prop}
\label{prop:min}
There exists a unique  closed orbit of $H$ in $\GB$ and 
it is the only minimal element of $\Gamma(G/H)$.  
\end{prop}

\begin{proof}
  Let $V_0$ be a closed orbit of $H$ in $\GB$.
Let $\H_0$ be the set of the $H$-orbit closures in $\GB$ 
linked with $V_0$ by an oriented  path in $\Gamma(G/H)$. 
It is sufficient to prove that  $\H_0=\H(\GB)$.

We assume that $\H(\GB)-\H_0$ is not empty.
Since $\GB$ belongs to  $\H_0$ and all the orbits 
are joined to $\GB$ by an oriented path, there exists
$Z\in \H(\GB)-\H_0$ and $\alpha\in\Delta$ such that 
$\alpha$ raises $Z$ to an element $Z'$ of $\H_0$ 
(it is sufficient to take $Z$ of maximal dimension in $\H(\GB)-\H_0$).
Let us fix such a pair $(Z,\alpha)$ such that $Z$ is of minimal dimension.

Since $Z'\neq V_0$, there exists $\beta\in\Delta$ and 
$Y\in \H_0$ such that $\beta$ raises $Y$ to $Z'$.
Since the edges of $\Gamma(G/H)$ are of type $U$ and
$Y\neq Z$, we have $\beta\neq\alpha$.

Using \cite[Lemma 3]{Br:GammaGH}, one easily checks that one of the two 
following graphs is a subgraph of $\Gamma(G/H)$:

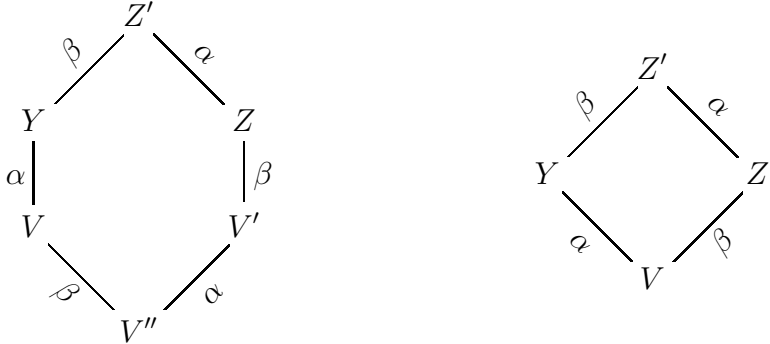
\begin{minipage}{0.4\textwidth}
  \begin{center}
\begin{diagram}[width=0.7cm,height=0.7cm]
  &             &Z' \\
  &\ldLine^\beta& &\rdLine^\alpha\\
Y &             & &              &Z\\
\dLine^\alpha&  & &              &\dLine_\beta\\
V &             & &              &V'\\
  &\rdLine_\beta& &\ruLine_\alpha\\
  &             &V''
  \end{diagram} 
  \end{center}
\end{minipage}
 \begin{minipage}{0.4\textwidth}
  \begin{center}
  \begin{diagram}[width=0.7cm,height=0.7cm]
  &             &Z'\\
  &\ruLine^\beta& &\rdLine^\alpha\\
 Y&             & &              &Z\\
  &\rdLine_\alpha& &\ruLine_\beta\\
  &             &V            
  \end{diagram}\end{center}\end{minipage}\\

In the first case, $Z$, $V'$ and $V''$ does not belong to $\H_0$.
By minimality of the dimension of $Z$ and by considering $(V'',\beta)$ 
we deduce that $V$ does not belong to $\H_0$.
Now, the pair $(V,\alpha)$ contradicts the minimality of the dimension of $Z$.
A similar argue works in the second case.
\end{proof}\\

Let $V$ in $\H(\GB)$ and $V_0$ denote the unique closed $H$-orbit $\GB$.
By Proposition~\ref{prop:min}, there exists an increasing path in 
$\Gamma(G/H)$ from $V_0$ to $V$. Let $(\a_1,\cdots,\alpha_s)$ be the sequence 
of the labels of the edges  of such a path.
Notice that $s=\dim(V)-\dim(V_0)$.
The inclusion relation $\H(\GB)$ can be read off the graph $\Gamma(G/H)$ 
by the following cancellation corollary:

\begin{coro}
\label{cor:cancel}
  We use above notation and fix $V'$ in $\H(\GB)$.
Set $k=\dim(V')-\dim(V_0)$.
Then, $V'\subset V$ if and only if there exists $i_1<\cdots<i_k$ 
such that the increasing path starting from $V_0$ and with labels
$(\a_{i_1},\cdots, \a_{i_k})$ ends at $V'$.
\end{coro}

\begin{proof}
  Using Proposition~\ref{prop:min}, the proof of \cite{RiSp1} works here.
\end{proof}

\parag
Let $d_G$ (resp. $d_H$) denote the dimension of the complete flag variety 
of $G$ (resp. $H$).
Then, we have the following ``symmetry'' on the set $\H(\GB)$:

\begin{prop}
\label{prop:sym}
Here we assume that $H$ is connected.
  For all $0\leq \delta\leq d_G-d_H$, the number of elements in 
$\H(\GB)$ of dimension $d_G-\delta$ equal those of dimension $d_H+\delta$.
\end{prop}

\begin{proof}
Consider $P_G(t)$ and $P_H(t)$ the Poincar\'e  polynomials of 
the complete flag varieties of $G$ and $H$.
By Poincar\'e duality, they are symmetric polynomials of degrees 
$d_G$ and $d_H$ ; that is,
$t^{d_G}P_G(1/t)=P_G(t)$ and $t^{d_H}P_H(1/t)=P_H(t)$.

Consider the following polynomial:
$$
Q(t)=\sum_{V\in{\H}(\GB)}t^{\dim(V)-d_H}.
$$
We claim that $Q(t).P_H(t)=P_G(t)$. 
The claim implies  that $Q(t)$ is symmetric and so the proposition. 

Let $\GB_H$ denote the flag variety of $H$.
For any $x\in\GB$, $H_x$ is a solvable subgroup of $H$ containing a maximal 
torus of $H$. It follows that $H_x$ is contained in a Borel subgroup of $H$:
consider $\varphi_x\::\:H.x\longto \GB_H$ the map induced  by this inclusion.
Moreover, the fiber $\varphi_x^{-1}(\varphi_x(x))$ is isomorphic to an affine 
space of dimension $\dim(H.x)-d_H$.

We choose one point in each orbit of $H$ in $\GB$ and consider the associated 
morphisms $\varphi_x$. 
There exists a finitely generated extension $K$ of $\QQ$ such that 
$G$, $H$, the inclusion of $H$ in $G$, the chosen points in $\GB$, the 
morphisms $\varphi_x$, the isomorphisms between the fibers the $\varphi_x$ and 
affine spaces are all defined. 
By taking an extension if necessary, we may (and shall) also assume that
the Schubert cells (for fixed Borel subgroups of $G$ and $H$) of $\GB$ and 
$\GB_H$ are defined and isomorphic to affine spaces over $K$.

Now, we consider a finite quotient $\FF_q$ of $K$ and the 
points $\GB(\FF_{q^n})$ of $\GB$ over $\FF_{q^n}$ for all positive integer $n$.
By using the decompositions of $\GB$ and $\GB_H$ in Schubert cells, one 
obtain:
$$
|\GB(\FF_{q^n})|=P_G(q^n)\ \ \ {\rm and}\ \ \ 
|\GB_H(\FF_{q^n})|=P_H(q^n).
$$ 
Now, we count the points in $\GB(\FF_{q^n})$ by using the decomposition 
in $H$-orbits:
$$
|\GB(\FF_{q^n})|=\sum_{V\in\H(\GB)}|V^\circ(\FF_{q^n})|=
\sum_{V\in\H(\GB)}|\GB_H(\FF_{q^n})|. (q^n)^{\dim(V)-d_H}=
P_H(q^n).Q(q^n).
$$
The claim follows.
\end{proof}

\subsection{Minimal rank and toroidal embeddings}

\parag
In this subsection, $(G,H)$ is a spherical pair not necessarily of minimal rank.
An {\it embedding of $G/H$} is a pair $(X,x)$ where $X$ is a normal 
and irreducible $G$-variety and $x$ is a point of $X$ such that 
$G.x$ is open in $X$ and $G_x=H$. 
Such an embedding is said to be {\it toroidal} if any irreducible 
$B$-stable divisor of $X$ which contains a $G$-orbit is $G$-stable.

\begin{lemma}
\label{lem:rkisot}
Let $G/H$ be a spherical homogeneous space (not necessarily of minimal
rank). Let $(X,x)$ be a toroidal embedding of $G/H$ and $y$ be a point in $X$.

Then, we have the inequality:
$$
\rk(G/H)+\rk(H)\geq \rk(G.y)+\rk(G_y).
$$  
In particular, if $G/H$ is of minimal rank, $G.y$ is.
\end{lemma}

\begin{proof}
Firstly, we prove that it is sufficient to show the lemma when 
$\dim(G.y)=\dim(X)-1$.
By \cite[Lemma 2.1.1]{BL}, since $X$ is toroidal, there exists $G$-orbits
$\O_0,\cdots,\O_s$ such that $G.y=\O_0\subset\overline{\O_1}\subset
\cdots\subset \overline{\O_s}=X $ (where $\overline{\O_i}$ denotes the closure
of $\O_i$ in $X$) and $\dim(\O_{i+1})=\dim(\O_i)+1$ for all $i=0,\cdots,s-1$.
For all $i$ we fix a point $y_i$ in $\O_i$.
Since $\overline{\O_i}$ is normal, we can apply the case when 
$\dim(G.y)=\dim(X)-1$ to each $\O_i\subset\overline{\O_{i+1}}$ showing that:
$$
\rk(G.y_{i+1})+\rk(G_{y_{i+1}})\geq \rk(G.y_i)+\rk(G_{y_i}).
$$
The inequality of the lemma follows.\\

We now assume that $\dim(G.y)=\dim(X)-1$. Set $\O=G.y$.
Consider the linear action of the group $G_y$ acts on quotient $T_yX/T_y\O$
of $T_yX$ by $T_y\O$. Since $X$ is normal, it is smooth at $y$ and 
$T_yX/T_y\O$ is a line. So, the action of $G_y$ defines a character 
$\chi\,:\,G_y\longto \KK^*$.

Let $T_y$ be a maximal torus of $G_y$. 
Let $S$ denote the neutral component of the kernel of the restriction 
of $\chi$ to $T_y$.
We claim that $S$ has fixed points in $G.x$.
Set $\Omega=G.x\cup\O$; it is open in $X$ and hence it is a smooth variety.
By a result of Bialynicki-Birula, we have 
$T_x(\Omega^S)=(T_x\Omega)^S$.
In particular, $\Omega^S$ is not contained in $\O$. 
This proves the claim.

By the claim, a subgroup conjugated to $S$ fixes $x$ and:
$$
\rk(H)\geq\dim(S)=\dim(T_y)-1=\rk(G_y)-1.
$$
Moreover, since $X$ is toroidal $\rk(G/H)=\rk(\O)+1$.
The lemma follows.
\end{proof}

\parag
The fixed points of a maximal torus of $G$ in the toroidal embeddings 
of spherical homogeneous spaces of minimal rank are easy to localize. 
Indeed, we have:

\begin{prop}
\label{prop:embcaract}
 Let  $(G,H)$ be a spherical pair and $T$ be a maximal torus of $G$.
The following are equivalent:
\begin{enumerate}
\item \label{ass:embcaract1}$(G,H)$ is of minimal rank.
\item \label{ass:embcaract2}There exists a complete toroidal embedding $(X,x)$ of $G/H$ such that
for all $x\in X^T$ $G.x$ is complete.
\item \label{ass:embcaract3}For all  complete toroidal embedding $(X,x)$ of $G/H$ and
for all $x\in X^T$,  $G.x$ is complete.
\end{enumerate}
\end{prop}

\begin{proof}
We assume that $(G,H)$ is of minimal rank and fix a complete toroidal 
embedding $(X,x)$ of $G/H$.
Let $y\in X^T$.
Lemma~\ref{lem:rkisot} shows that $\rk(G.y)=0$; that is $G.y$ is complete.
This proves that  Assertion~\ref{ass:embcaract1} implies 
Assertion~\ref{ass:embcaract3}.\\

Conversely, let $(X,x)$ satisfying Assertion~\ref{ass:embcaract2}.
It remains to prove that $(G,H)$ is of minimal rank.

Let $\lambda$ be a one-parameter subgroup of $T$ such that 
$T$ is the centralizer of the image of $\lambda$ 
(that is, $\lambda$ is regular) and $X^\lambda=X^T$ (where $X^\lambda$ 
denote the set of fix points of the image of $\lambda$).
Since $\lambda$ is regular, 
the set $g\in G$ such that $\lim_{t\to 0} \lambda(t)g\lambda(t^{-1})$ 
exists in $G$
is a Borel subgroup of $G$ denoted by $B(\lambda)$.
By Proposition~\ref{prop:defequivHGB}, it is sufficient to prove that for all
$y\in G.x$ we have $\rk(B(\lambda).y)=\rk(G/H)$.
This holds by Lemma~\ref{lem:torrk} below since the rank of a complete $G$-orbit 
equals zero.
\end{proof}

\begin{lemma}
\label{lem:torrk}
Let $(X,x)$ be a complete toroidal embedding of the spherical homogeneous
space $G/H$.
Let $y$ be a point in the open $G$-orbit in $X$
Let $\lambda$ be a regular one-parameter subgroup of $T$ such 
that $X^\lambda=X^T$.
Set $z=\lim_{t\to 0}\lambda(t)y$.

Then, we have:
$$
\rk(G/H)-\rk(B(\lambda).y)=\rk(G.z).
$$
\end{lemma}

\begin{proof}
Let us introduce some material and notation from \cite{BL}.
There exists a parabolic subgroup $P$ of $G$ containing $T$ 
such that $P_z$ is reductive.
Let   $Q$ denote the parabolic subgroup of $G$ containing $T$ and opposite 
to $P$.
We have $G_z^u\subset Q^u$ and there exists a closed subvariety 
$A\subset Q^u$ $P_z$-stable such that the product in $G$ induces an isomorphism
$A\times Q^u_z\longto Q^u$.
By \cite[Lemma 1.1]{BL}, there exists a locally closed affine normal and 
irreducible subvariety $S$ of $X$ such that $S\cap G.z=\{z\}$, $S$ is 
$P_z$-stable and the morphism $G\times S\longto X$ induced by the action 
is smooth at $(1,z)$.
In particular, we have:
\begin{eqnarray}
  \label{eq:dimS}
  \dim(S)=\dim(G/H)-\dim(G.z)=\rk(G/H)-\rk(G.z).
\end{eqnarray}
Let $P\times_{P_z}(A\times S)$ denote the quotient of $P\times A\times S$ 
by the action of $P_z$ defined by
$
p.(q,a,s)=(qp^{-1},pap^{-1},ps),
$
where $p\in P_z$, $(q,a,s)\in P\times A\times S$.
The group $P$ acts naturally on this variety.
By \cite[Theorem 5]{BL}, the morphism
$$
\begin{array}{cccl}
\Theta\ :&P\times_{P_z}(A\times S)&\longto&X\\
&(p:(a,s))&\longmapsto&(pa).s
\end{array}
$$
is an open immersion.

Consider the Bialynicki-Birula cell 
$$
X(\lambda,z)=\{p\in X\ :\ \lim_{t\to 0}\lambda(t)p=z\}.
$$
Notice that $y\in X(\lambda,z)$.
By \cite[Propositions 2.1 and 2.3]{BL}, 
$X(\lambda,z)\cap G.x=B(\lambda).y$ and $G.x\cap S=T.y'$ for some 
$y'\in B(\lambda).y$.
Then, the proof of \cite[Proposition 2.3]{BL} shows that 
$\Theta$ induces by restriction an isomorphism:
$$
\bigg ( P\cap B(\lambda)\bigg )\times_{B(\lambda)_z}
\bigg ( (A\cap B(\lambda))\times T.y'\bigg )
\longto B(\lambda).y.
$$
Since $T$ is contained in $B(\lambda)_z$, this isomorphism implies that 
\begin{eqnarray}
  \label{eq:rkBy}
  \rk(B(\lambda).y)=\dim(T.y')=\dim S.
\end{eqnarray}
The lemma follows from Equalities~\ref{eq:dimS}~and~\ref{eq:rkBy}.
\end{proof}

\section{Reduction to the case when $G$ and $H$ are semi-simple}
\label{sec:red}

The goal of this section is to reduce the classification of the spherical pairs
$(G,H)$ to those with $G$ semi-simple adjoint and $H$ semi-simple.

\begin{prop}
\label{prop:redHred}
  Let $(G,H)$ be a spherical pair of minimal rank.

Then, there exists a parabolic subgroup $P$ of $G$ 
with a Levi decomposition $P=P^uL$ and a reductive subgroup 
$K$ of $L$ such that:
\begin{enumerate}
\item $H=P^uK$
\item $(L,K)$ is a spherical pair of minimal rank.
\end{enumerate}
\end{prop}

\begin{proof}
  We can write $H=H^uK$, where $K$ is a reductive subgroup of $H$.
But, by [Humphreys, 30.3], there exists a parabolic subgroup 
$P=P^uL$ of $G$ such that $H^u\subseteq P^u$ and $K\subseteq L$.
We claim that $P$ and $L$ satisfy the proposition.\\

Let firstly prove that $P^u=H^u$.

Let $T(H)$ be a maximal torus of $K$ (and hence of $H$).
The variety $\GB$ contains on open subset stable by $P$ (and hence by $H$) 
isomorphic to $P^u\times \GB(L)$ (with obvious notation). 
By assumption, there exists a point $x$ in $\GB$ fixed by $T(H)$ such that 
$H.x$ is open in $\GB$.
But, $x=(u,y)$ belongs to $P^u\times \GB(L)$.
Since the $H$-orbit of $x$ is open in  $P^u\times \GB(L)$, 
so is its intersection with $P^u\times \{y\}$. 
Hence, the set of the $hkuk^{-1}\in P^u$ such that $h\in H^u$ and $k\in K_y$ 
is open and dense in $P^u$.
In particular, the $K_y$-orbit $K_y.(uH^u)$ is open and dense in $P^u/H^u$.

Since $x\in\GB^{T(H)}$, $uH^u\in (P^u/H^u)^{T(H)}$.
But, $K_y$ is a solvable group with $T(H)$ as maximal torus. 
So, $K_y^\circ.uH^u$ is one orbit of the unipotent 
radical $K_y^u$ of $K_y^\circ$.
In particular, it is closed in the affine variety $P^u/H^u$.
But it is also open. 
We  deduce that $K_y$ acts transitively on $P^u/H^u$.

But $K_y^u$ is contained is $K$ and normalizes $H^u$. 
So, $H^u$ is a fix point of $K_y$ in $P^u/H^u$. 
We deduce that $P^u/H^u$ has only one point ; that is, that $P^u=H^u$.\\

On the other hand, $K.y$ is open in $\GB(L)$ and $y$ is fixed by 
the maximal torus $T(H)$ of $K$. We deduce that $(L,K)$ is a spherical pair of minimal rank.
\end{proof}\\

Since the parabolic subgroups of a given reductive group are very well
known, Proposition~\ref{prop:redHred} reduces the problem of 
classification of the spherical pairs $(G,H)$ of minimal rank to 
the case when $H$ is reductive.

\begin{prop}
\label{prop:redHssc}
Let $G$ be a connected reductive group.
Set $G_{\rm ad}=G/Z(G)$ and 
consider the projection $p\,:\,G\longto G_{\rm ad}$.
Let $H$ be a reductive subgroup of $G$. 
Then,
\begin{enumerate}
\item 
\label{ass:redHssc1}
The pair $(G,H)$ is spherical of minimal rank if and only if 
the pair $(G_{\rm ad},p(H))$ is.
\item \label{ass:redHssc2} 
The pair $(G,H)$ is spherical of minimal rank if and only if 
the pair $(G,H^\circ)$ is.
\item \label{ass:redHssc3}
If $G/H$ is of minimal rank, 
the neutral component $p(H)^\circ$ of $p(H)$ is semi-simple.
\end{enumerate}
\end{prop}

\begin{proof}
Assertions~\ref{ass:redHssc1}~and~\ref{ass:redHssc2} are obvious from 
Assertion~\ref{ass:defequivHGB0} of Proposition~\ref{prop:defequivHGB}.

To prove the last assertion, it is sufficient to prove that the connected
center $S$ of $H$ is contained in the center of $G$.
There exists $x\in\GB$ fixed by $S$ such that $H.x$ is open in $\GB$.
Since, $H\subset G^S$, $G^S.x$ is open in $\GB$.
But, $G^S.x$ is an irreducible component of $\GB^S$.
Therefore, $\GB^S=\GB$ and $S$ is central in $G$.
\end{proof}\\

Proposition~\ref{prop:redHssc} reduces the problem of 
classification of the spherical pairs $(G,H)$ of minimal rank to 
the case when $G$ is semi-simple adjoint and $H$ is semi-simple.
From now on, we only consider such pairs.

\section{Classification of Lie algebras}
\label{sec:class}

Let $(G,H)$ be a spherical pair of minimal rank  with
$G$ semi-simple adjoint and $H$ semi-simple.
Let $\lg$ (resp. $\lh$) denote the Lie algebra of $G$ (resp. $H$).

\subsection{Root systems of $\lg$ and $\lh$}

Let $T(H)$ be a maximal torus of $H$. Let $T\supset T(H)$ be a maximal torus 
of $G$. 
Let $\Chi(T)={\rm Hom}(T,\KK^*)$ (resp. $\Chi(T(H))={\rm Hom}(T(H),\KK^*)$) 
denote the character group of $T$ (resp. $T(H)$). 
Let $\phi_\lg\subset \Chi(T)$ (resp. $\phi_\lh\subset \Chi(T(H))$) 
be the set of roots  of $\lg$ (resp. $\lh$).
Let $\rho\,:\,\Chi(T)\longto\Chi(T(H))$ be the restriction map.

In this subsection, we will prove some very constraining relations between
$\phi_\lg$, $\phi_\lh$ and $\rho$.

\subsubsection{}\label{sec:phigh}
The following stability of the set spherical pairs of minimal rank will be
used to localize the study over some fixed roots of $\lh$:

\begin{lemma}
\label{lem:loc}
  Let $S$ be a subtorus of $H$.

Then, $(G^S,\,H^S)$ is a spherical pair of minimal rank. 
\end{lemma}

\begin{proof}
Let $T(H)$ be a maximal torus of $H$ which contains $S$. 
  Let $x$ be a fixed point of $T(H)$ in $\GB$ such that $H.x$ 
is open in $\GB$.
Since $V\cap G^S.x$ is open in $G^S.x$, it is irreducible. 
So, it is an irreducible component of $V^S$.
Now, \cite{Rich} implies that $V\cap G^S.x=H^S.x$.
In particular, $H^S.x$ is open in $G^S.x\simeq \GB(G^S)$ and
$x$ is fixed by the maximal torus $T(H)$ of $H^S$. 
The lemma follows. 
\end{proof}

\begin{lemma}
\label{lem:rhophi}
With the above notation, we have 
$\rho(\phi_\lg)= \phi_\lh$. 
\end{lemma}

\begin{proof}
  Let $\a\in\phi_\lg$.
Set $S={\rm Ker}(\rho(\a))^\circ\subset T(H)$.
By Lemma~\ref{lem:loc}, $H^S$ is a spherical subgroup of $G^S$ and 
$\rk(G^S/H^S)=\rk(G^S)-\rk(H^S)$. Since the semi-simple rank of $G^S$
is one, this implies that ${H^S}^\circ$ is not a torus.
So, $\rho(\a)$ is a root of $\lh$.

Moreover, since $\lh\subset\lg$, $\phi_\lh\subset\rho(\phi_\lg)$.
\end{proof}\\

By Lemma~\ref{lem:rhophi}, we can define the map
$\rhobar\,:\,\phi_\lg\longto\phi_\lh,\,\a\longmapsto\rho(\a)$.

\begin{lemma}
\label{lem:hsl2}
The spherical pairs $(G,H)$ of minimal rank with 
$G$  semi-simple adjoint, $H$ connected and  
$\lh=\lsl_2$ are:  

\begin{enumerate}
\item $(\PSL_2,\PSL_2)$.
\item $\PSL_2$ diagonally embedded in $\PSL_2\times\PSL_2$.
\end{enumerate}
\end{lemma}

\begin{proof}
  By Assertion~\ref{ass:defequivHGB0} of Proposition~\ref{prop:defequivHGB}, 
the dimension of $\GB$ is at most 2. 
We deduce that $G=\PSL_2$ or $\PSL_2\times\PSL_2$.
The lemma follows easily.
\end{proof}

\medskip

\begin{lemma}
\label{lem:fibrerho}
Let $\a\in\phi_\lh$.

Then, $\rhobar^{-1}(\a)$ contains  either one root  of $\lg$ or 
two orthogonal roots of $\lg$.
Moreover, if 
$\rhobar^{-1}(\a)=\{\a^\circ\}$ then $\lh_\a=\lg_{\a^\circ}$; 
and if $\rhobar^{-1}(\a)=\{\a^-,\a^+\}$ with $\a^-\neq\a^+\in\phi_\lg$
then $\lh_\a\neq\lg_{\a^\pm}$.
\end{lemma}

\begin{proof}
  Set $S=\ker(\a)^\circ$.
Since $\lh^S\simeq \lsl_2$ and by Lemma~\ref{lem:loc}, 
we can apply Lemma~\ref{lem:hsl2} to $(G^S/S,H^S/S)$.
The lemma follows immediately.
\end{proof}\\

Lemma~\ref{lem:fibrerho} divides the set of roots of $\lh$ in two kinds:

\begin{center}
\begin{tabular*}{0.8\textwidth}{{@{\extracolsep{\fill}}ccc}}
$
\phi_\lh^1:=\{\a\in\phi_\lh\;:\;|\rhobar^{-1}(\a)|=1\}
$ &and &
$
\phi_\lh^2:=\{\a\in\phi_\lh\;:\;|\rhobar^{-1}(\a)|=2\}.
$
\end{tabular*}
\end{center}

We denote by $W_H$ the Weyl group $N_H(\TH)/\TH$ of $H$.

\begin{lemma}
\label{lem:Winv}
  The sets $\phi_\lh^1$ and $\phi_\lh^2$ are stable by the action of $W_H$.
\end{lemma}

\begin{proof}
By Lemma~\ref{lem:loc}, $(G^\TH,\TH)$ is a spherical pair of minimal rank.
So, $G^\TH$ is a torus and $G^\TH=T$.
In particular, $N_H(\TH)$ is contained in $N_G(T)$; this inclusion
induces an injection of $W_H=N_H(\TH)/\TH$ into $W=N_G(T)/T$.
By this injection, we obtain an action of $W_H$ on $\Chi(T)$ such that 
$\rho$ is $W_H$-equivariant.
The lemma follows.
\end{proof}

\subsubsection{Simple roots}

In Section~\ref{sec:phigh}, we just proved that $\rho$ induces a map from
$\phi_\lh$ onto $\phi_\lh$. In this section, we will prove that $\rho$
induces a map from the Dynkin diagram of $\lg$ onto those of $\lh$.

Let us fix a choice $\phi_\lh^+$ of positive roots for $\lh$.
Set $\phi_\lg^+=\rhobar^{-1}(\phi_\lh^+)$. 
Let $\Delta_\lg$ (resp. $\Delta_\lh$) be the set of simple roots of $\phi_\lg$
(resp. $\phi_\lh$).

\begin{lemma}
\label{lem:simpleroot}
Let $\a$ be a root of $\lg$.
Then, $\a\in \Delta_\lg$ if and only if $\rhobar(\a)\in\Delta_\lh$.
\end{lemma}

\begin{proof}
Since $\a\in\phi_\lg^+$ if and only if   $\rho(\a)\in\phi_\lh^+$, 
we may assume that $\a\in\phi_\lg^+$.\\

Let us assume that $\a\not\in\Delta_\lg$.
Then, there exists $\beta$ and $\gamma$ in $\phi_\lg^+$ such that
$\a=\beta+\gamma$.
By applying $\rho$, we see that $\rho(\a)$ does not lie in $\Delta_\lh$.\\

Let us assume that $\a\in\Delta_\lg$.
By absurd, we assume that there exists $\beta$ and $\gamma$ in 
$\phi_\lh^+$ such that $\rho(\a)=\beta+\gamma$.
Three cases occurs:

\noindent
\underline{Case 1}: $\beta$ and $\gamma$ belong to $\phi_\lh^1$.

Let $\beta^\circ$ and $\gamma^\circ$ be in $\phi_\lg$ such that 
$\rho(\beta^\circ)=\beta$ and $\rho(\gamma^\circ)=\gamma$.
By Lemma~\ref{lem:hsl2}, 
$\lg_{\beta^\circ}=\lh_\beta$ and $\lg_{\gamma^\circ}=\lh_\gamma$.
So, we have:
$$
[\lg_{\beta^\circ},\,\lg_{\gamma^\circ}]=
[\lh_\beta,\,\lh_\gamma]=\lh_{\rho(\a)}.
$$
In particular this bracket is non zero and $\beta^\circ+\gamma^\circ$ is 
a root of $\lg$.
Moreover, $\lg_{\beta^\circ+\gamma^\circ}=\lh_{\rho(\a)}$. 
So, Lemma~\ref{lem:hsl2} shows that $\rho(\a)\in\phi_\lh^1$.
But $\a$ and $\beta^\circ+\gamma^\circ$ belong to 
$\rhobar^{-1}(\rho(\a))$. 
So, $\a=\beta^\circ+\gamma^\circ$; and this root is  not simple.

\noindent
\underline{Case 2}: $\beta\in\phi_\lh^1$ and $\gamma\in\phi_\lh^2$.

Let $\beta^\circ$ as above. 
We can write $\rhobar^{-1}(\gamma)=\{\gamma^+,\,\gamma^-\}$.
We have:
$$
\lg_{\beta^\circ+\gamma^+}+\lg_{\beta^\circ+\gamma^-}\supset
[\lg_{\beta^\circ},\,\lg_{\gamma^+}+\lg_{\gamma^-}]\supset
[\lg_{\beta^\circ},\,\lh_\gamma]=
[\lh_{\beta},\,\lh_\gamma]=
\lh_{\beta+\gamma}=\lh_{\rho(\a)}.
$$
Moreover, since $\lh_\gamma$ is different from $\lg_{\gamma^+}$ and 
$\lg_{\gamma^-}$, $\lh_{\rho(\a)}$ is different 
from $\lg_{\beta^\circ+\gamma^+}$ and $\lg_{\beta^\circ+\gamma^-}$.
In particular, $\beta^\circ+\gamma^+$ and $\beta^\circ+\gamma^-$ are roots of 
$\lg$; 
and $\rhobar^{-1}(\rho(\a))=\{\beta^\circ+\gamma^+,\,\beta^\circ+\gamma^-\}$. 
So, $\a=\beta^\circ+\gamma^+$ or $\beta^\circ+\gamma^-$; and this root is not simple.

\noindent
\underline{Case 3}: $\beta$ and $\gamma$ belong to $\phi_\lh^2$.

With obvious notation, we have:
$$
\lg_{\beta^++\gamma^+}+
\lg_{\beta^++\gamma^-}+
\lg_{\beta^-+\gamma^+}+
\lg_{\beta^-+\gamma^-}
\supset
[\lh_\beta,\,\lh_\gamma]=\lh_{\beta+\gamma}.
$$

If $\lh_{\beta+\gamma}$ equals one of the four spaces 
$\lg_{\beta^\pm +\gamma^\pm}$, Lemma~\ref{lem:fibrerho} shows that 
$\alpha$ equals $\beta^\pm +\gamma^\pm$ and is not a simple root.
Else, two of the four spaces $\lg_{\beta^\pm +\gamma^\pm}$ are not zero and
$\alpha$ equals one of the two corresponding roots; in particular 
 $\a$ is not simple.
\end{proof}\\

Consider the map
$$
\begin{array}{cccc}
\rhodelta\,:&\Delta_\lg&\longto&\Delta_\lh\\
&\a&\longmapsto&\rhobar(\alpha).
\end{array}
$$
Set $\Delta_\lh^2=\Delta_\lh\cap\phi_\lh^2$ and  
$\Delta_\lh^1=\Delta_\lh\cap\phi_\lh^1$.

On the Dynkin diagram of $\lh$, we color in black the simple roots 
in $\Delta^2\lh$. The so obtained diagram is called the {\it colored
Dynkin diagram of $\lh$} and is denoted by $\dyn_\lh$.
From now on, when we draw the Dynkin diagram $\dyn_\lg$ of $\lg$,
two simple roots $\a$ and $\beta$ are placed on the same vertical line
if and only if $\rhodelta(\a)=\rhodelta(\beta)$; in such a way,
$\rho_\Delta$ identifies with the vertical projection.
Note that by Lemma~\ref{lem:fibrerho}, $\a$ and $\beta$ are orthogonal.
For $({\rm PSL}_4,{\rm PSp_4})$, we obtain the following picture:

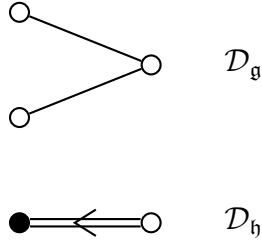
\begin{center}
  \begin{pspicture}(-1,-3.5)(3.5,0.5)
\psset{labelsep=0.4}
\cnode(0,0){0.2}{A}
\cnode(2.5,-1){0.2}{B}
\cnode(0,-2){0.2}{C}
\ncline{A}{B}
\ncline{B}{C}
\uput[r](3.5,-1){$\dyn_\lg$}

\cnode[fillstyle=solid](0,-4){0.2}{D}
\cnode(2.5,-4){0.2}{E}
\ncline[offset=0.07]{E}{D}
\ncline[offset=-0.07]{E}{D}
\ncline[nodesep=1pt,linecolor=white]{D}{E}\lput{:U}{{\supdiag}}
\uput[r](3.5,-4){$\dyn_\lh$}
  \end{pspicture}
\end{center}

By exchanging the simple roots in a fiber of $\rho_\Delta$, we define an 
involution $\sigma_\lh$ on the set of vertices of the Dynkin diagram of $\lg$.

\subsection{A result of unicity}

\begin{prop}
\label{prop:unicity}
For a fixed pair $(\dyn_\lg,\sigma_\lh)$ there exists at most one
(up to conjugacy by an element of $G$)  
spherical pair $(G,H)$ where $G$ is adjoint and $H$ connected.
\end{prop}

\begin{proof}
Obviously, $G$ is determined by $\dyn_\lg$.
Let us fix a Borel subgroup $B$ of $G$ and a maximal torus $T$ of $B$.
Let $\Delta_\lg$ denote the set of simple roots of $G$. 
For any $\a\in\Delta_\lg$, we fix a $\lsl_2$-triple $(X_\a,Y_\a,H_\a)$.
Consider 
$$
\begin{array}{cccl}
\Theta\,:&T&\longto&(\KK^*)^{\Delta_\lh^2}\\
&t&\longmapsto&(\beta(t)\alpha(t^{-1}))_{
  \begin{array}{l}
\alpha\neq\beta\in\Delta_\lg\\
\rhodelta(\alpha)=\rhodelta(\beta)
  \end{array}
}.
\end{array}
$$
The neutral component $S$ of the kernel of $\Theta$ 
is a subtorus of $T$ of dimension 
$|\Delta_\lg|-|\Delta^2_\lh|=|\Delta_\lh|$.
Moreover, $\Theta$ is surjective.

Let $H$ be a semi-simple subgroup of $G$ such that $(G,H)$ is a 
spherical pair of minimal rank with $(\dyn_\lg,\dyn_\lh,\rhodelta)$ 
as associated triple.
Let $\TH$ be a maximal torus of $G$.
Up to conjugacy, we may assume that $\TH$ is contained in $T$.
But, $\TH$ is contained in $S$; by an argument of dimension,
we conclude that $\TH=S$.

For all $\alpha\in\Delta_\lg^1$, we have
$\lg_\a=\lh_{\rho(\a)}$. Moreover, 
we claim that up to conjugacy, we may assume that  
for all  $\alpha\neq\beta\in\Delta_\lg$ 
such that $\rhodelta(\alpha)=\rhodelta(\beta)$ we have
$\lh_{\rho(\a)}=\KK.(X_\a+X_\beta)$.

We write 
$\Delta^2_\lh=\{\a_1,\cdots,\a_k\}$ and
$\Delta^2_\lg=\{\a_1^-,\cdots,\a_k^-\}\cup\{\a_1^+,\cdots,\a_k^+\}$
such that for all $i=1,\cdots,k$, $\rho(\a_i^\pm)=\alpha_i$.
By Lemma~\ref{lem:fibrerho}, there exist $(x_1,\cdots,x_k)\in\KK^*$ such that 
for all $i=1,\cdots,k$, 
$\lh_{\a_i}=\KK.(X_{\a_i^-}+x_i X_{\a_i^+})$.
Since $\Theta$ is surjective, there exists $t\in T$ such that 
$\Theta(t)=(x_1,\cdots,x_k)$.
By conjugating $H$ by $t$, we obtain the claim.

Let $i\in\{1,\cdots,k\}$.
There exists $y\in\KK^*$ such that 
$\lh_{-\a_i}=\KK.(Y_{\a_i^-}+y Y_{\a_i^+})$.
Since $\a_i^-$ and $\a_i^+$ are orthogonal, 
$\xi:=[X_{\a_i^-}+X_{\a_i^+},Y_{\a_i^-}+y Y_{\a_i^+}]=
H_{\a_i^-}+yH_{\a_i^+}$.
But, $\xi$ belongs to the Lie algebra of $\TH$, that is to $S$, so
$(\a_i^--\a_i^+)(\xi)=0$. We conclude that $y=1$.

Finally, since $\lh$ is generated as Lie algebra by the $\lh_{\pm\alpha}$ for $\alpha\in\Delta_\lh$; $\lh$  is generated by:
$$
\{X_\alpha\}_{\alpha\in\Delta^1_\lg}\cup \{Y_\alpha\}_{\alpha\in\Delta^1_\lg}
\cup 
\{X_{\a_i^-}+X_{\a_i^+}\}_{\alpha\in\Delta^2_\lh}\cup
\{Y_{\a_i^-}+Y_{\a_i^+}\}_{\alpha\in\Delta^2_\lh}.
$$
In particular, $\lh$ only depends on the triple 
$(\dyn_\lg,\sigma_\lh)$.
\end{proof}

\subsubsection{The case when $\rk(H)=2$}

Lemma~\ref{lem:hsl2} considers the case when $\rk(H)=1$. 
We now consider the case when $\rk(H)=2$:

\begin{lemma}
\label{lem:rkH2}
\psset{unit=0.5cm}
We assume that $\rk(H)=2$.
Then, the possibilities for $(\dyn_\lg,\sigma_\lh)$ are:
\begin{enumerate}
\item $(\dyn_\lh,{\rm Identity})$ obtained with $G=H$.
\item $(\dyn_\lh\cup\dyn_\lh,{\rm Exchange})$  obtained with $H$
embedded diagonally in $H\times H$.
\item  
$\dyn_\lg=A_3$ and $\sigma_\lh$ fixes the central vertex and exchanges the two 
others; obtained with
$H={\rm PSp}_4\subset G={\rm PSL}_4$:

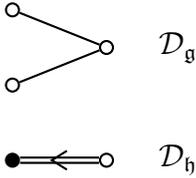
\begin{pspicture}(-1,-4.5)(3.5,0.5)
\psset{labelsep=0.4}
\cnode(0,0){0.2}{A}
\cnode(2.5,-1){0.2}{B}
\cnode(0,-2){0.2}{C}
\ncline{A}{B}
\ncline{B}{C}
\uput[r](3.5,-1){$\dyn_\lg$}

\cnode[fillstyle=solid](0,-4){0.2}{D}
\cnode(2.5,-4){0.2}{E}
\ncline[offset=0.07]{E}{D}
\ncline[offset=-0.07]{E}{D}
\ncline[nodesep=1pt,linecolor=white]{D}{E}\lput{:U}{{\supdiag}}
\uput[r](3.5,-4){$\dyn_\lh$} 
  \end{pspicture}\\ 

\item \label{diag:G2SO7}
$\dyn_\lg=B_3$ and $\sigma_\lh$ fixes the central vertex and exchanges the two 
others; obtained with
$G_2\subset {\rm SO}_7$:

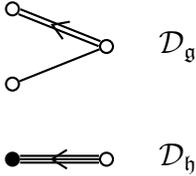
\begin{pspicture}(-1,-4.5)(3.5,0.5)
\psset{labelsep=0.4}
\cnode(0,0){0.2}{A}
\cnode(2.5,-1){0.2}{B}
\cnode(0,-2){0.2}{C}
\ncline[offset=0.07]{A}{B}
\ncline[offset=-0.07]{A}{B}
\ncline[nodesep=1pt,linecolor=white]{A}{B}\lput{:U}{{\supdiag}}
\ncline{B}{C}
\uput[r](3.5,-1){$\dyn_\lg$}

\cnode[fillstyle=solid](0,-4){0.2}{D}
\cnode(2.5,-4){0.2}{E}
\ncline[offset=0.09]{E}{D}
\ncline[offset=-0.09]{E}{D}
\ncline{D}{E}\lput{:U}{{\supdiag}}
\uput[r](3.5,-4){$\dyn_\lh$} 
  \end{pspicture}\\ 

\end{enumerate}
\end{lemma}
\begin{proof}
If $\phi_\lh^2$ is empty, the dimensions of $\lg$ and $\lh$ are equal and hence
$\lg=\lh$. From now on, we assume that $\phi_\lh^2$ is not empty.
Since $\phi_\lh^2$ is stable by the action of the Weyl group $W_H$ of $H$, 
$\Delta^2_\lh$ is non empty.

By invariance by $W_H$ the possibilities for colored Dynkin diagrams
of $\lh$ are:
$$
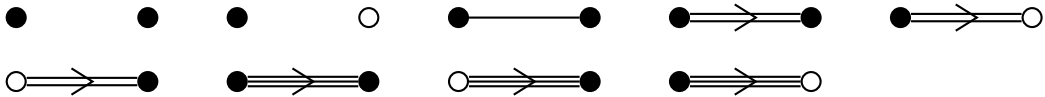
\begin{array}{ccccc}

  \begin{pspicture}(-0.5,-0.5)(3,0.5)
\cnode[fillstyle=solid](0,0){0.2}{A}   
\cnode[fillstyle=solid](2.5,0){0.2}{B}
  \end{pspicture}
&
\begin{pspicture}(-0.5,-0.5)(3,0.5)
\cnode[fillstyle=solid](0,0){0.2}{A}   
\cnode(2.5,0){0.2}{B}
  \end{pspicture}
&
\begin{pspicture}(-0.5,-0.5)(3,0.5)
\cnode[fillstyle=solid](0,0){0.2}{A}   
\cnode[fillstyle=solid](2.5,0){0.2}{B}
\ncline{A}{B}
  \end{pspicture}&
\begin{pspicture}(-0.5,-0.5)(3,0.5)
\cnode[fillstyle=solid](0,0){0.2}{A}   
\cnode[fillstyle=solid](2.5,0){0.2}{B}
\ncline[offset=0.07]{A}{B}
\ncline[offset=-0.07]{A}{B}
\ncline[nodesep=1pt,linecolor=white]{B}{A}\lput{:U}{{\supdiag}}
  \end{pspicture}
  &\begin{pspicture}(-0.5,-0.5)(3,0.5)
\cnode[fillstyle=solid](0,0){0.2}{A}   
\cnode(2.5,0){0.2}{B}
\ncline[offset=0.07]{A}{B}
\ncline[offset=-0.07]{A}{B}
\ncline[nodesep=1pt,linecolor=white]{B}{A}\lput{:U}{{\supdiag}}
  \end{pspicture}\\
  \begin{pspicture}(-0.5,-0.5)(3,0.5)
\cnode(0,0){0.2}{A}   
\cnode[fillstyle=solid](2.5,0){0.2}{B}
\ncline[offset=0.07]{A}{B}
\ncline[offset=-0.07]{A}{B}
\ncline[nodesep=1pt,linecolor=white]{B}{A}\lput{:U}{{\supdiag}}
  \end{pspicture}&
\begin{pspicture}(-0.5,-0.5)(3,0.5)
\cnode[fillstyle=solid](0,0){0.2}{A}   
\cnode[fillstyle=solid](2.5,0){0.2}{B}
\ncline[offset=0.09]{A}{B}
\ncline[offset=-0.09]{A}{B}
\ncline{B}{A}\lput{:U}{{\supdiag}}
  \end{pspicture}
&
 \begin{pspicture}(-0.5,-0.5)(3,0.5)
\cnode(0,0){0.2}{A}   
\cnode[fillstyle=solid](2.5,0){0.2}{B}
\ncline[offset=0.09]{A}{B}
\ncline[offset=-0.09]{A}{B}
\ncline{B}{A}\lput{:U}{{\supdiag}}
  \end{pspicture} & \begin{pspicture}(-0.5,-0.5)(3,0.5)
\cnode[fillstyle=solid](0,0){0.2}{A}   
\cnode(2.5,0){0.2}{B}
\ncline[offset=0.09]{A}{B}
\ncline[offset=-0.09]{A}{B}
\ncline{B}{A}\lput{:U}{{\supdiag}}
  \end{pspicture}
\end{array}
$$

In each case, using the action of $W_H$, one can determine
$\phi_\lh^1$ and $\phi_\lh^2$ and thus, compute the cardinality
$|\phi_\lg|$ of $\phi_\lg$ which equals $|\phi^1_\lh|+2|\phi_\lh^2|$.
Moreover, by Lemma~\ref{lem:hsl2} the two simple roots of $\lg$ 
which map on a simple root of $\phi_\lh^2$ are orthogonal.
These constraints imply a small number of possibilities for 
$(\dyn_\lg,\dyn_\lh,\rhodelta)$.
In the following tabular, we list  these possibilities omitting those 
corresponding to $\lh\subset\lh\times\lh$
and $\lsl_2\times\lsl_2\to\lsl_2\times\lsl_2\times\lsl_2,
(\xi,\eta)\mapsto(\xi,\xi,\eta)$.

$$
\begin{array}{|c|c|c|c|c|c|c|}
\hline
{\rm Case}&
\lh&
\begin{array}{c}
{\rm Colored}\\{\rm Dynkin\ diagram}\\
{\rm of\ }\lh
\end{array}&
(|\phi^1_\lh|,|\phi_\lh^2|)&
|\phi_\lg|&
\begin{array}{c}
{\rm Dynkin\ diagram}\\{\rm of\ }\lg
\end{array}&
\lg\\

\hline
1&\lsl_3&
\begin{pspicture}(-0.5,-0.5)(3,0.5)
\cnode[fillstyle=solid](0,0){0.2}{A}
\cnode[fillstyle=solid](2.5,0){0.2}{B}
\ncline{A}{B}\end{pspicture}&
(0,6)&12 &\begin{pspicture}(-0.5,-2.5)(3,0.5)
\cnode(0,0){0.2}{A}\cnode(2.5,0){0.2}{B}
\cnode(0,-2){0.2}{C}\cnode(2.5,-2){0.2}{D}
\ncline{C}{D}
\ncline[offset=0.07]{A}{B}\ncline[offset=-0.07]{A}{B}
\end{pspicture}
&
\lsp_4\times\lsl_2\times\lsl_2\\
\hline
2&\lsp_4&
\begin{pspicture}(-0.5,-0.5)(3,0.5)
\cnode[fillstyle=solid](0,0){0.2}{A}   
\cnode[fillstyle=solid](2.5,0){0.2}{B}
\ncline[offset=0.07]{A}{B}
\ncline[offset=-0.07]{A}{B}
  \end{pspicture}&
(0,8)&16 &\begin{pspicture}(-0.5,-2.5)(3,0.5)
\cnode(0,0){0.2}{A}\cnode(2.5,0){0.2}{B}
\cnode(0,-2){0.2}{C}\cnode(2.5,-2){0.2}{D}
\ncline[offset=0.07]{A}{B}\ncline[offset=-0.07]{A}{B}
\ncline[offset=0.07]{C}{D}\ncline[offset=-0.07]{C}{D}
\ncline[nodesep=1pt,linecolor=white]{B}{A}\lput{:U}{{\supdiag}}
\ncline[nodesep=1pt,linecolor=white]{C}{D}\lput{:U}{{\supdiag}}
  \end{pspicture}
&\lsp_4\times\lsp_4
\\
\hline
3&\lsp_4&
\begin{pspicture}(-0.5,-0.5)(3,0.5)
\cnode[fillstyle=solid](0,0){0.2}{A}\cnode(2.5,0){0.2}{B}
\ncline[offset=0.07]{A}{B}
\ncline[offset=-0.07]{A}{B}
\ncline[nodesep=1pt,linecolor=white]{B}{A}\lput{:U}{{\supdiag}}
 \end{pspicture}&
(4,4)&12 &\begin{pspicture}(-1,-2.5)(3.5,0.5)
\psset{labelsep=0.4}
\cnode(0,0){0.2}{A}
\cnode(2.5,-1){0.2}{B}
\cnode(0,-2){0.2}{C}
\ncline{A}{B}
\ncline{B}{C}
  \end{pspicture}
&\lsl_4
\\
\hline
4&\lsp_4&
\begin{pspicture}(-0.5,-0.5)(3,0.5)
\cnode[fillstyle=solid](0,0){0.2}{A}\cnode(2.5,0){0.2}{B}
\ncline[offset=0.07]{A}{B}
\ncline[offset=-0.07]{A}{B}
\ncline[nodesep=1pt,linecolor=white]{A}{B}\lput{:U}{{\supdiag}}
 \end{pspicture}&
(4,4)&12 &\begin{pspicture}(-1,-2.5)(3.5,0.5)
\psset{labelsep=0.4}
\cnode(0,0){0.2}{A}
\cnode(2.5,-1){0.2}{B}
\cnode(0,-2){0.2}{C}
\ncline{A}{B}
\ncline{B}{C}
  \end{pspicture}
&\lsl_4
\\
\hline
5&G_2&\begin{pspicture}(-0.5,-0.5)(3,0.5)
\cnode[fillstyle=solid](0,0){0.2}{A}\cnode[fillstyle=solid](2.5,0){0.2}{B}
\ncline[offset=0.09]{A}{B}\ncline[offset=-0.09]{A}{B}\ncline{B}{A}
 \end{pspicture}
&
(0,12)&24 &\begin{pspicture}(-0.5,-2.5)(3,0.5)
\cnode(0,0){0.2}{A}\cnode(2.5,0){0.2}{B}
\cnode(0,-2){0.2}{C}\cnode(2.5,-2){0.2}{D}
\ncline[offset=0.07]{A}{B}\ncline[offset=-0.07]{A}{B}
\ncline[nodesep=1pt,linecolor=white]{A}{B}\lput{:U}{{\supdiag}}
\ncline{C}{D}\ncline{A}{D}
  \end{pspicture}
&
\lsp_8
\\
\hline
6&G_2&
\begin{pspicture}(-0.5,-0.5)(3,0.5)
\cnode[fillstyle=solid](0,0){0.2}{A}\cnode(2.5,0){0.2}{B}
\ncline[offset=0.09]{A}{B}\ncline[offset=-0.09]{A}{B}
\ncline{B}{A}\lput{:U}{{\supdiag}}
 \end{pspicture}
&
(6,6)&18 &\begin{pspicture}(-1,-2.5)(3.5,0.5)
\psset{labelsep=0.4}
\cnode(0,0){0.2}{A}
\cnode(2.5,-1){0.2}{B}
\cnode(0,-2){0.2}{C}
\ncline[offset=0.07]{A}{B}\ncline[offset=-0.07]{A}{B}
\ncline[nodesep=1pt,linecolor=white]{A}{B}\lput{:U}{{\supdiag}}
\ncline{B}{C}
  \end{pspicture}
&\lso_7
\\
\hline
7&G_2&\begin{pspicture}(-0.5,-0.5)(3,0.5)
\cnode[fillstyle=solid](0,0){0.2}{A}\cnode(2.5,0){0.2}{B}
\ncline[offset=0.09]{A}{B}\ncline[offset=-0.09]{A}{B}
\ncline{A}{B}\lput{:U}{{\supdiag}}
 \end{pspicture}
&
(6,6)&18 &
\begin{pspicture}(-1,-2.5)(3.5,0.5)
\psset{labelsep=0.4}
\cnode(0,0){0.2}{A}
\cnode(2.5,-1){0.2}{B}
\cnode(0,-2){0.2}{C}
\ncline[offset=0.07]{A}{B}\ncline[offset=-0.07]{A}{B}
\ncline[nodesep=1pt,linecolor=white]{A}{B}\lput{:U}{{\supdiag}}
\ncline{B}{C}
  \end{pspicture}&\lso_7
\\
\hline
\end{array}
$$ 

There is no nontrivial action of the group ${\rm PSL}_3$  on $\PP^1$.
But in Case~1, $\PP^1$ is a factor of $\GB$; and $H$ cannot have a dense 
orbit in $\GB$.
In Case~2, the projections of $\lh$ on each factor of $\lg$ are isomorphisms:
Case~2 also cannot occur.

Consider Case~3.
Let $\a$ (resp. $\beta$) denote the short (resp. long) simple root of $\lh$.
Set $\a^\circ=\rhodelta^{-1}(\alpha)$ and
$(\alpha+\beta)^\circ=\rhodelta^{-1}(\alpha)$.
By Lemma~\ref{lem:Winv}, $\alpha+\beta\in\phi_\lh^1$. So,
$$
\lh_{2\alpha+\beta}=[\lh_\alpha,\lh_{\a+\beta}]=
[\lg_{\alpha^\circ},\lg_{(\alpha+\beta)^\circ}]=
\lg_{\alpha^\circ+(\alpha+\beta)^\circ}.
$$
Now, Lemma~\ref{lem:fibrerho} shows that $2\alpha+\beta\in\phi_\alpha^1$.
With Lemma~\ref{lem:Winv}, this contradicts $\beta\in\phi^2_\lh$.

By elimination, the inclusion of ${\rm PSp_4}$ in ${\rm PSL}_4$ corresponds 
to Case~4.

Consider Case~5.
We label the simple roots of $\lg$ and $\lh$ as follows:
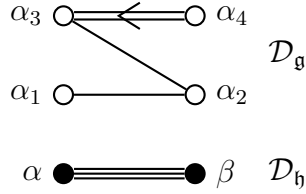
\begin{center}
  \begin{pspicture}(-1,-3.5)(3.5,0.5)
\cnode(0,0){0.2}{A}
\uput[l](0,0){$\alpha_3$}
\cnode(2.5,0){0.2}{B}
\uput[r](2.5,0){$\alpha_4$}
\cnode(0,-1.5){0.2}{C}
\uput[l](0,-1.5){$\a_1$}
\cnode(2.5,-1.5){0.2}{D}
\uput[r](2.5,-1.5){$\a_2$}
\ncline[offset=0.07]{A}{B}\ncline[offset=-0.07]{A}{B}
\ncline[nodesep=1pt,linecolor=white]{A}{B}\lput{:U}{\supdiag}
\ncline{A}{D}\ncline{C}{D}
\uput[r](3.5,-0.75){$\dyn_\lg$}

\cnode[fillstyle=solid](0,-3){0.2}{E}
\uput[l](0,-3){$\a$}
\cnode[fillstyle=solid](2.5,-3){0.2}{F}
\uput[r](2.5,-3){$\beta$}
\ncline[offset=0.09]{E}{F}\ncline[offset=-0.09]{E}{F}
\ncline{E}{F}
\uput[r](3.5,-3){$\dyn_\lh$}
  \end{pspicture}
\end{center}

We have 
$\rho(\a_3+\a_4)=\rho(\a_2+\a_1)=\rho(\a_2+\a_3)=\a+\beta$; 
but $\a_3+\a_4$, $\a_2+\a_1$ and $\a_2+\a_3$ are three distinct roots 
of $\lg$. This contradicts Lemma~\ref{lem:fibrerho}.

Consider Case~6.
Let $\a$ (resp. $\beta$) denote the short (resp. long) simple root of $\lh$.
By Lemma~\ref{lem:Winv}, $\beta+2\alpha$ belongs to $\phi^1_\lh$.
By the argument used in Case~3 before, one easily checks that 
$\beta+3\alpha=(\beta+2\alpha)+\alpha$ belongs to $\phi^1_\lh$.
This contradicts Lemma~\ref{lem:Winv}, since $\beta+3\alpha$ is a long root.

Case~7 corresponds to the inclusion of $G_2$ in ${\rm SO}_7$.
\end{proof}

\medskip
We may now assume that $H$ is simple. Indeed, we have:                            
\begin{prop}
\label{prop:redHsimple}
Let $(G,H)$ be a spherical pair of minimal rank with $G$ semi-simple adjoint 
and $H$ connected.
If $H$ is not simple then there exists two spherical pairs
$(G_1,H_1)$ and $(G_2,H_2)$ of minimal rank such that 
$G=G_1\times G_2$ and $H=H_1\times H_2$.
\end{prop}

\begin{proof}
By assumption, $\dyn_\lh$ is the disjoint union of two Dynkin
diagrams $\dyn_1$ and $\dyn_2$.
By Lemmas~\ref{lem:rkH2}~and~\ref{lem:loc}, 
for all $\alpha,\,\beta\in\Delta_\lg$ such that $\rho(\alpha)$ 
and $\rho(\beta)$ are orthogonal, $\alpha$ and $\beta$ are orthogonal.
We deduce that $\dyn_\lg$ is the disjoint union of 
$\rhodelta^{-1}(\dyn_1)$ and $\rhodelta^{-1}(\dyn_2)$.
The proposition follows.
\end{proof}\\

By Proposition~\ref{prop:redHsimple}, to classify all the 
spherical pairs $(G,H)$ of minimal rank with $G$ simple-simple adjoint 
and $H$ semi-simple, we may assume that $H$ is simple.
The Theorem~\ref{th:introclass} stated in the introduction lists all such
spherical pairs.
We can now prove this classification.\\


\begin{proof}[of Theorem~\ref{th:introclass}]
  By Proposition~\ref{prop:unicity}, it is sufficient to classify the 
possible triples $(\dyn_\lg,\dyn_\lh,\rhodelta)$.
By  Lemma~\ref{lem:rkH2}, we may assume that $\rk(H)\geq 3$.
Moreover, we may assume that $\Delta_\lh^2$ is non empty and different
from $\Delta_\lh$.
Let $\a\in\Delta_\lh^2$ and $\beta\in\Delta_\lh^1$.
By Lemma~\ref{lem:rkH2}, either $\a$ and $\beta$ are orthogonal or
$\alpha$ is the short root joined to the long root $\beta$ by a 
double edge.
One easily deduces that the colored Dynkin diagram of $\lh$ is one of the 
following:
$$
\psset{labelsep=0.4}
\begin{array}{l}
 \begin{pspicture}(-1,-0.5)(10.5,0.5)
\cnode[fillstyle=solid](0,0){0.2}{J}
\cnode[fillstyle=solid](2.5,0){0.2}{K}
\cnode[fillstyle=solid](5,0){0.2}{L}
\cnode[fillstyle=solid](7.5,0){0.2}{M}\cnode(10,0){0.2}{N}
\ncline{J}{K}
\ncline{L}{M}
\ncline[linestyle=dotted]{K}{L}
\ncline[offset=0.07]{M}{N}
\ncline[offset=-0.07]{M}{N}
\ncline[nodesep=1pt,linecolor=white]{M}{N}\lput{:U}{{\supdiag}}
  \end{pspicture}\\
 \begin{pspicture}(-1,-0.5)(10.5,0.5)
\cnode(0,0){0.2}{J}\cnode(2.5,0){0.2}{K}
\cnode(5,0){0.2}{L}\cnode(7.5,0){0.2}{M}
\cnode[fillstyle=solid](10,0){0.2}{N}
\ncline{J}{K}
\ncline{L}{M}
\ncline[linestyle=dotted]{K}{L}
\ncline[offset=0.07]{M}{N}
\ncline[offset=-0.07]{M}{N}
\ncline[nodesep=1pt,linecolor=white]{N}{M}\lput{:U}{{\supdiag}}
\end{pspicture}\\
\begin{pspicture}(-1,-0.5)(8,0.5)
\cnode(0,0){0.2}{G}\cnode(2.5,0){0.2}{H}
\cnode[fillstyle=solid](5,0){0.2}{I}
\cnode[fillstyle=solid](7.5,0){0.2}{J}

\ncline{G}{H}
\ncline{I}{J}
\ncline[offset=0.07]{H}{I}
\ncline[offset=-0.07]{H}{I}
\ncline[nodesep=1pt,linecolor=white]{I}{H}\lput{:U}{{\supdiag}}
  \end{pspicture}
\end{array}
$$

One easily deduces from Lemmas~\ref{lem:loc}~and~\ref{lem:rkH2} than
in the three preceding case the Dynkin diagram $\dyn_\lg$ is respectively:
$$
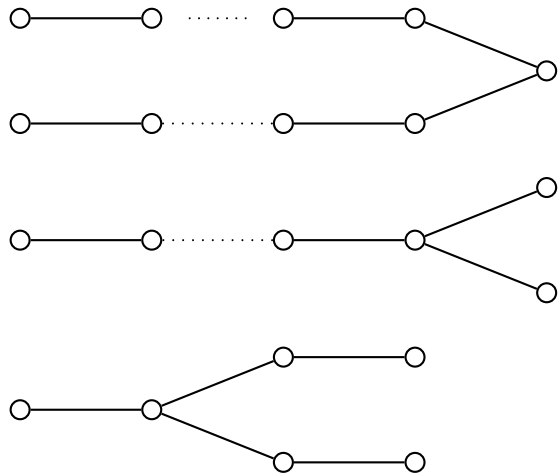
\begin{array}{l}
 \begin{pspicture}(-1,-2.5)(10.5,0.5)
\psset{labelsep=0.4}
\cnode(0,0){0.2}{A}\cnode(2.5,0){0.2}{B}
\cnode(5,0){0.2}{C}\cnode(7.5,0){0.2}{D}

\cnode(10,-1){0.2}{E}

\cnode(0,-2){0.2}{I}\cnode(2.5,-2){0.2}{H}
\cnode(5,-2){0.2}{G}\cnode(7.5,-2){0.2}{F}

\ncline{A}{B}
\ncline[nodesep=0.5,linestyle=dotted]{B}{C}
\ncline{C}{D}
\ncline{D}{E}
\ncline{E}{F}
\ncline{F}{G}
\ncline[linestyle=dotted]{G}{H}
\ncline{H}{I}
  \end{pspicture}\\
 \begin{pspicture}(-1,-2.5)(10.5,0.5)
\cnode(0,-1){0.2}{A}\cnode(2.5,-1){0.2}{B}
\cnode(5,-1){0.2}{C}\cnode(7.5,-1){0.2}{D}

\cnode(10,0){0.2}{E}\cnode(10,-2){0.2}{F}

\ncline{A}{B}
\ncline[linestyle=dotted]{B}{C}
\ncline{C}{D}
\ncline{D}{E}
\ncline{D}{F}
  \end{pspicture}\\
\begin{pspicture}(-1,-2.5)(8,0.5)
\cnode(0,-1){0.2}{A}\cnode(2.5,-1){0.2}{B}
\cnode(5,0){0.2}{C}\cnode(7.5,0){0.2}{D}
\cnode(5,-2){0.2}{E}\cnode(7.5,-2){0.2}{F}

\ncline{A}{B}
\ncline{B}{C}
\ncline{C}{D}
\ncline{B}{E}
\ncline{E}{F}
  \end{pspicture}

\end{array}
$$
The theorem follows.
\end{proof}

\bibliographystyle{smfalpha}
\bibliography{biblio}

\def\cprime{$'$}
\providecommand{\bysame}{\leavevmode ---\ }
\providecommand{\og}{``}
\providecommand{\fg}{''}
\providecommand{\smfandname}{\&}
\providecommand{\smfedsname}{\'eds.}
\providecommand{\smfedname}{\'ed.}
\providecommand{\smfmastersthesisname}{M\'emoire}
\providecommand{\smfphdthesisname}{Th\`ese}
\begin{thebibliography}{Kno95}

\bibitem[BJ06]{BrionJoshua:rngmin}
{\scshape M.~Brion {\normalfont \smfandname} R.~Joshua} -- {\og Equivariant
  cohomology and chern classes of symmetric varieties of minimal rank\fg},
  \emph{Preprint} (2006), no.~www.math.ohio-state.edu/\~{}joshua/ec.pdf,
  p.~1--20.

\bibitem[BL87]{BL}
{\scshape M.~Brion {\normalfont \smfandname} D.~Luna} -- {\og Sur la structure
  locale des vari\'et\'es sph\'eriques\fg}, \emph{Bull. Soc. Math. France}
  \textbf{115} (1987), no.~2, p.~211--226.

\bibitem[Bri01]{Br:GammaGH}
{\scshape M.~Brion} -- {\og On orbit closures of spherical subgroups in flag
  varieties\fg}, \emph{Comment. Math. Helv.} \textbf{76} (2001), no.~2,
  p.~263--299.

\bibitem[Bri04]{Br:eqvectbun}
{\scshape M.~Brion} -- {\og Construction of equivariant vector bundles\fg},
  \emph{Preprint} (2004), no.~arXiv:math/0410039, p.~1--21.

\bibitem[Kno95]{Kn:WGH}
{\scshape F.~Knop} -- {\og On the set of orbits for a {B}orel subgroup\fg},
  \emph{Comment. Math. Helv.} \textbf{70} (1995), no.~2, p.~285--309.

\bibitem[Pin01]{Spin:these}
{\scshape S.~Pin} -- {\og Sur les singularit\'es des orbites d'un sous-groupe
  de {B}orel dans les espaces sym\'etriques\fg}, \emph{Th\`ese, Universit\'e
  Grenoble I} (2001), p.~1--109, {\tt
  http://www-fourier.ujf-grenoble.fr/THESE/these\_daterev.html}.

\bibitem[Res04]{GammaGH}
{\scshape N.~Ressayre} -- {\og Sur les orbites d'un sous-groupe sph\'erique
  dans la vari\'et\'e des drapeaux\fg}, \emph{Bull. de la SMF} \textbf{132}
  (2004), p.~543--567.

\bibitem[Res05]{actionKnop}
\bysame , {\og About {K}nop's action of the {W}eyl group on the set of orbits
  of a spherical subgroup in the flag manifold\fg}, \emph{Transform. Groups}
  \textbf{10} (2005), no.~2, p.~255--265.

\bibitem[Ric82]{Rich}
{\scshape R.~W. Richardson} -- {\og On orbits of algebraic groups and {L}ie
  groups\fg}, \emph{Bull. Austral. Math. Soc.} \textbf{25} (1982), no.~1,
  p.~1--28.

\bibitem[RS90]{RiSp1}
{\scshape R.~W. Richardson {\normalfont \smfandname} T.~A. Springer} -- {\og
  The {B}ruhat order on symmetric varieties\fg}, \emph{Geom. Dedicata}
  \textbf{35} (1990), no.~1-3, p.~389--436.

\bibitem[Tch02]{Tchou}
{\scshape A.~Tchoudjem} -- {\og Cohomologie des fibr\'es en droites sur la
  compactification magnifique d'un groupe semi-simple adjoint\fg}, \emph{C. R.
  Math. Acad. Sci. Paris} \textbf{334} (2002), no.~6, p.~441--444.

\bibitem[Tch05]{Tchoudjem:rngmin}
\bysame , {\og Cohomologie des fibr\'{e}s en droites sur les vari\'{e}t\'{e}s
  magnifiques de rang minimal\fg}, \emph{Preprint} (2005), no.~Arxiv :
  AG/0507581, p.~1--48.

\end{thebibliography}

\begin{center}
  -\hspace{1em}$\diamondsuit$\hspace{1em}-
\end{center}

\vspace{5mm}
\begin{flushleft}
Nicolas Ressayre\\
Universit{\'e} Montpellier II\\
D{\'e}partement de Math{\'e}matiques\\
Case courrier 051-Place Eug{\`e}ne Bataillon\\
34095 Montpellier Cedex 5\\
France\\
e-mail:~{\tt ressayre@math.univ-montp2.fr}  
\end{flushleft}
\end{document}